\documentclass[12pt,a4paper,reqno]{amsart}
\usepackage{amsmath,amsfonts,amssymb,amsthm,amscd}
\usepackage[english]{babel}
\usepackage[mathscr]{eucal}
\usepackage{graphicx}

\renewcommand{\a}{\alpha}
\renewcommand{\b}{\beta}
\newcommand{\g}{\gamma}

\newcommand{\n}{\nu}
\newcommand{\om}{\omega}
\newcommand{\OM}{\Omega}

\newcommand{\s}{\sigma}

\newcommand{\e}{\varepsilon}
\newcommand{\f}{\varphi}
\newcommand{\F}{\Phi}

\newcommand{\p}{\psi}

\newcommand{\R}{{\mathbb R}}

\newcommand{\RR}{{\mathbb R}^2}

\newcommand{\curl}{{\rm curl}\,}

\newcommand{\diver}{{\rm div}\,}

\newcommand{\pd}{\partial}

\newcommand{\dt}{\partial_t}
\newcommand{\lllll}{L^{\infty}\left(\R^+,L^1\cap L^{\infty}(\RR)\right)}
\newcommand{\llll}{L^1\cap L^{\infty}(\RR)}
\newcommand{\eps}{\varepsilon}

\newcommand{\supp}{\operatorname{supp\,}}

\newcommand{\loc}{\operatorname{{loc}}}

\newtheorem{theorem}{Theorem}[section]
\newtheorem{proposition}[theorem]{Proposition}
\newtheorem{lemma}[theorem]{Lemma}
\newtheorem{corollary}[theorem]{Corollary}

\newtheorem{definition}[theorem]{Definition}

\theoremstyle{remark}
\newtheorem{remark}[theorem]{Remark}

\numberwithin{equation}{section}

\newcommand{\na}{\nabla}
\newcommand{\vti}{\tilde{v}}

\begin{document}


\title[Uniqueness for the vortex-wave system]{Uniqueness for the vortex-wave system when the vorticity is constant near the point vortex}
\author[C. Lacave $\&$ E. Miot]{Christophe Lacave $\&$ Evelyne Miot}

\address[C. Lacave]{Universit\'e de Lyon\\
Universit\'e Lyon1\\
INSA de Lyon, F-69621\\
Ecole Centrale de Lyon\\
CNRS, UMR 5208 Institut Camille Jordan\\
Batiment du Doyen Jean Braconnier\\
43, blvd du 11 novembre 1918\\
F - 69622 Villeurbanne Cedex\\
France} \email{lacave@math.univ-lyon1.fr}

\address[E. Miot]{Laboratoire J.-L. Lions UMR 7598,
Universit\'e Pierre et Marie Curie,
175 rue du Chevaleret, 75013 Paris,
France} \email{miot@ann.jussieu.fr}

\date{\today}

\begin{abstract}
We prove uniqueness for the vortex-wave system with a single point vortex introduced by Marchioro and Pulvirenti
\cite{mar_pul} in the case where the vorticity is initially constant near the point vortex.
 Our method relies on the Eulerian approach for this problem and in particular
 on the formulation in terms of the velocity.
\end{abstract}

\maketitle

\section{Introduction}

\label{introduction}

In this paper, we study a system occurring in two dimensional fluid
dynamics. The motion of an ideal incompressible fluid in $\RR$ with
divergence-free velocity field $v=(v_1,v_2):\R^+ \times  \RR
\rightarrow \RR$ and vorticity $\om=\curl v=\partial_1
v_2-\partial_2 v_1:\R^+ \times \RR \rightarrow \R$ is given by the
Euler equations
\begin{equation}
\label{i1}
\begin{cases}
\dt \om + v\cdot \nabla \om=0,\\
\om=\curl v,\:\diver v=0,
\end{cases}
\end{equation}
where $\diver v=\partial_1 v_1 + \partial_2 v_2$. For this system,
Yudovich's Theorem states global existence and uniqueness in
$\lllll$ for an initial vorticity $\om_0 \in \llll$. Equation
\eqref{i1} is a transport equation with field $v$, therefore one may
solve it with the method of characteristics. When $v$ is smooth, it
gives rise to a flow defined by
\begin{equation}
\label{i2}
\begin{cases}
\frac{d}{dt} \phi_t(x)=v\big(t,\phi_t(x)\big) \\
\phi_0(x)=x \in \RR.
\end{cases}
\end{equation}
In view of (\ref{i1}), we then have
\begin{equation}
\label{i3} \frac{d}{dt} \om\big(t,\phi_t(x)\big)\equiv 0,
\end{equation}
which means that $\om$ is constant along the characteristics. In the
general case of a vorticity $\om \in \lllll$, these computations may
be rigorously justified, so that the Eulerian formulation \eqref{i1}
and the Lagrangian one \eqref{i2}, \eqref{i3} turn out to be
equivalent.

Since equation \eqref{i1} governs the evolution of the vorticity
$\om$, it is natural to express the velocity
 $v$ in terms of $\om$. This can be done by taking
the orthogonal gradient in both terms in the relation $\om=\curl v$
and using that $v$ is divergence free. This yields $\na^\perp \om
=\Delta v$, so that under the additional constraint that $v$
vanishes at infinity, we have
\begin{equation}
\label{i4} v=K\ast \om.
\end{equation}
Here $\ast$ denotes the convolution product and $K : \RR\setminus
\{0\} \rightarrow \RR$ stands for the Biot-Savart Kernel defined by
\begin{equation}
\label{i5}
K(x)=\frac{1}{2\pi} \frac{x^{\bot}}{|x|^2},\qquad x\neq 0,
\end{equation}
where $ (x_1,x_2)^{\bot}=(-x_2,x_1). $ When the  vorticity tends to
be concentrated at points, one may modify equation \eqref{i1}
according to formulas \eqref{i4} and \eqref{i5} into a system of
ordinary differential equations, called point vortex system, which
governs the motion of these points. A rigorous justification for
this system has been carried out in \cite{article2}. It is proved
there that if the initial vorticity $\om_0$ is close to the weighted
sum of Dirac masses $\sum d_i \delta_{z_i}$ in a certain sense, then
$\om(t)$ remains close to $\sum d_i \delta_{z_i(t)}$ for all time,
where the vortices $z_i(t)$ evolve according to the point vortex
system.

In the early 90s, Marchioro and Pulvirenti
\cite{mar_pul,livrejaune} investigated the mixed problem in which
the vorticity is composed of an $L^\infty$ part and a sum of Dirac
masses. They obtained the so-called vortex-wave system, which couples
the usual point vortex system and the classical Lagrangian
formulation for the two-dimensional fluid dynamics. In the case of a
single point vortex (which will be the case studied here), these
authors obtained the global existence of solutions of the
vortex-wave system in Lagrangian formulation.
\begin{definition}[Lagrangian solutions]
\label{defLAGR} Let $\om_0 \in \llll$ and $z_0\in \RR$. We say that
the triple $(\om,z,\phi)$ is a global Lagrangian solution to the
vortex-wave system with initial condition $(\om_0,z_0)$ if $ \om \in
\lllll$, $v=K\ast \om \in C(\R^+ \times \RR)$ and
\begin{equation*}
z:\R^+ \to \RR,\qquad \phi:\R^+\times \RR\setminus \{z_0\} \to
\RR
\end{equation*}
are such that $z\in C^1(\R^+,\RR)$, $\phi(\cdot,x) \in
C^1(\R^+,\RR)$ for all $x\neq z_0$ and satisfy
\begin{equation}\tag{LF}
\begin{cases}
v(\cdot,t)=(K\ast \omega)(\cdot,t),\\
 \dot{z}(t)=v(t,z(t)),\\
 z(0)=z_0,\\
\dot{\,\phi}_t(x)=v(t,\phi_t(x))+
K\big(\phi_t(x)-z(t)\big),\\
\phi_0(x)=x, \; x\neq z_0 ,\\
\omega(\phi_t(x),t)=\omega_0(x),
\end{cases}
\label{mixte}
\end{equation}
where $\phi_t=\phi(t,\cdot)$. In addition, for all $t$, $\phi_t$ is an
homeomorphism from $\RR\setminus\{z_0\}$ into
$\RR\setminus\{z(t)\}$ that preserves Lebesgue's measure.
\end{definition}

This system involves two kinds of trajectories.
 The point vortex $z(t)$ moves under the velocity field $v$ produced by the
regular part $\om$ of the vorticity. This regular part and the
vortex point give rise to a smooth flow $\phi$ along which $\om$ is
constant. The main difference with the classical Euler dynamics is
the presence of the field $K(x-z(t))$, which is singular at the
point vortex but smooth elsewhere. Marchioro and Pulvirenti
\cite{mar_pul} proved global existence for \eqref{mixte}. The proof
mainly relies on estimates involving the distance between
$\phi_t(x)$ and $z(t)$ and uses almost-Lipschitz regularity for
$v=K\ast \om$ and the explicit form of $K$. It is shown in
particular that a characteristic starting far apart from the point
vortex cannot collide with $z(t)$ in finite time. Consequently, the
singular term $K(\phi_t(x)-z(t))$ in \eqref{mixte} remains
well-defined for all time.

The notion of Lagrangian solutions is rather strong. One can define
a weaker notion of solutions: solutions in the sense of distributions
of the PDE (without involving the flow $\phi$). We call these
Eulerian solutions and we define them here below.

\begin{definition}[Eulerian solutions]
\label{defPDE}
 Let $\om_0\in L^1\cap L^\infty(\RR)$ and $z_0\in \RR$.
We say that $(\om,z)$ is a global Eulerian solution of the
vortex-wave equation with initial condition $(\om_0,z_0)$ if
\begin{equation*}
\om\in \lllll,\qquad z\in C(\R^+,\RR)
\end{equation*}
and if we have in the sense of distributions
\begin{equation}\tag{EF}
\begin{cases}
\label{t1} \dt \om +\diver((v+H) \om)=0, \\
 \om(0)=\om_0, \\
\dot{z}(t)=v\big(t,z(t)\big),\qquad z(0)=z_0,
\end{cases}
\end{equation}
where $v$ and $H$ are given by
\begin{equation*}
v(t,\cdot)=K\ast_{x} \om(t),\qquad H(t,\cdot)=K(\cdot-z(t)).
\end{equation*}
In other words, we have \footnote{By virtue of Lemma
\ref{loglip}, the field defined by $v=K\ast \om$ belongs to
$L^\infty(\R^+ \times \RR)$. On the other hand, $H$ belongs to
$L^1_{\loc}(\R^+ \times \RR)$, so that this definition makes sense.} for any test function $\p \in \mathcal{D}(\R^+\times \RR)$
\begin{equation*}
\begin{split}
 -\int_{\RR} \om_0(x) \p(0,x)\,dx&=\int_{\R^+}
\int_{\RR} \om(\dt \p+ (v+H)\cdot \nabla \p)\,ds\,dx,
\end{split}
\end{equation*}
and \footnote{We will see in Lemma \ref{loglip} and Proposition \ref{prop : cont-velocity} that $v(t)$ is defined for all time and is continuous in the space variable.}
\begin{equation*}
z(t)=z_0+\int_0^t v(s,z(s))\,ds
\end{equation*} for all $t\in \R^+$.
\end{definition}

This kind of Eulerian solutions appears for example in \cite{ift_lop}. In that paper, a solution of the Euler equation with a fixed point vortex is obtained as the limit of the Euler equations in the exterior of an obstacle that shrinks to a point. The regularity of the limit solution obtained in \cite{ift_lop} is not better than the one given in Definition \ref{defPDE}.

In this paper, we are concerned with the problems of uniqueness of Eulerian and Lagrangian solutions and with the related question of equivalence of Definitions \ref{defLAGR} and  \ref{defPDE}.

We will first prove the following Theorem, clarifying that a Lagrangian solution is an Eulerian solution.

\begin{theorem}
\label{thmequivalence}
 Let $\om_0\in L^1\cap L^\infty(\RR)$ and
$z_0\in \RR$. Let $(\om,z,\phi)$ be a global Lagrangian solution of
the vortex-wave system with initial condition $(\om_0,z_0)$. Then
$(\om,z)$ is a global Eulerian solution.
\end{theorem}

We turn next to our main purpose and investigate uniqueness for Lagrangian or Eulerian solutions.

Uniqueness for Lagrangian solutions can be easily achieved when the
support of $\om_0$ does not meet $z_0$; in that case, the support of
$\om(t)$ never meets $z(t)$ and the field $x\mapsto
K(\phi_t(x)-z(t))$ is Lipschitz on $\textrm{supp}\ \om_0$.

Another situation that has been studied is the case where the
vorticity is initially constant near the point vortex. Marchioro and
Pulvirenti \cite{mar_pul} suggested with some indications that
uniqueness for Lagrangian solutions should hold in that situation. This was proved by
Starovoitov \cite{russe} under the supplementary assumption that $\om_0$ is Lipschitz. In
this paper, we treat the general case where the initial vorticity is
constant near the point vortex $z_0$ and belongs to $\llll$. More
precisely, we prove the following

\begin{theorem}
\label{theo}  Let $\om_0\in L^1\cap L^\infty(\RR)$ and $z_0\in \RR$
such that there exists $R_0>0$ and $\a \in\R$ such that
\begin{equation*}
\om_0 \equiv \a  \: \: \textrm{on }\: B(z_0,R_0).
\end{equation*}
Suppose in addition that  $\om_0$ has compact support. Then there
exists a unique Eulerian solution of the vortex-wave system with
this initial data.
\end{theorem}

In order to prove Theorem \ref{theo}, we first show that if
$(\om,z)$ is an Eulerian solution, then $\om$ is a renormalized
solution in the sense of DiPerna-Lions \cite{dip-li} of its
transport equation. This in turn implies that if the vorticity is
initially constant near the point vortex, then this remains true for
all time. We then take advantage of the weak formulation \eqref{t1}
to derive a partial differential equation satisfied by the velocity
$v=K\ast \om$. In order to compare two solutions, one not only has
to compare the two regular parts, but also possibly the diverging
trajectories of the two vortices. Given two Eulerian solutions
$(\om_1,z_1)$ and $(\om_2,z_2)$, we therefore introduce the quantity
\begin{equation*}
r(t)=|\tilde{z}(t)|^2+\|\tilde{v}(t)\|_{L^2(\RR)}
\end{equation*}
where $\tilde{z}=z_1-z_2$, $\tilde{\om}=\om_1-\om_2$ and
$\tilde{v}=v_1-v_2=K\ast \tilde{\om}$. Since $\tilde{\om}$ vanishes
in a neighborhood of the point vortex, the velocity $\tilde{v}$ has
to be harmonic in this neighborhood. This provides in particular a
control of  its $L^{\infty}$ norm (as well as the $L^{\infty}$ norm
for the gradient) by its $L^2$ norm, which ultimately yields a
Gronwall-type estimate for $r(t)$ and allows to prove that it
vanishes.

\medskip

Finally, although we have chosen to restrict our attention to Eulerian solutions, we point out in Section \ref{final} that the renormalization property established for the linear transport equation  can be used to show the converse of Theorem \ref{thmequivalence}. This implies that Definitions \ref{defLAGR} and \ref{defPDE} are equivalent for any $\om_0 \in L^1\cap L^\infty(\RR)$, even if the vorticity is not initially constant in a neighborhood of the point vortex.

\section{Lagrangian implies Eulerian}

\label{section2}

We first briefly recall some remarkable properties of the convolution by the Biot-Savart Kernel $K$. The proofs are standard and may be found in \cite{maj-bert,livrejaune}. We begin with the Calder\'on-Zygmund
inequality.

\begin{lemma}
\label{calderon} Let $f \in \llll$ and $g=K\ast f$ so that $\curl g=f$ and $\diver g=0$. Then for all
$2\leq p<+\infty$, we have
\begin{equation*}
\|\nabla g\|_{L^p(\RR)} \leq Cp \|f\|_{L^p(\RR)},
\end{equation*}
where $C$ is some universal constant.
\end{lemma}
The following Lemma will be very useful in our further analysis.
\begin{lemma}
\label{loglip} Let $f \in \llll$ and $g=K\ast f$. Then $g$ satisfies
\begin{eqnarray}
\label{linfini_bound}
 \|g\|_{L^{\infty}} \leq
C\big(\|f\|_{L^{\infty}}+\|f\|_{L^1}\big).
\end{eqnarray}
Moreover,
\begin{eqnarray}
\label{loglipschitz} |g(x)-g(y)|\leq C(
\|f\|_{L^{\infty}},\|f\|_{L^1})\, \f\big( |x-y|\big),\ \forall
x,y\in \RR,
\end{eqnarray}
where $\varphi$ is the continuous, concave and non-decreasing
function defined by
\begin{eqnarray*}
\varphi(z) =\left\{\begin{array}{ll}z(1-\ln(z))&\textrm{if}\qquad 0\leq z<1 \\
1 &\textrm{if} \qquad z\geq 1.
\end{array}\right.
\end{eqnarray*}
\end{lemma}

From now on, we will denote by $\mathcal{AL}$ the set of almost-Lipschitz functions from $\RR$ into $\RR$, that is those for which
$$
|g(x)-g(y)|\leq C\f(|x-y|),\qquad x,y\in \RR
$$
 with some constant $C$, and by
$L^\infty(\mathcal{AL})$ the set of functions $v=v(t,x):\R \times
\RR \mapsto \RR$ satisfying
$$
|v(t,x)-v(t,y)|\leq C\f(|x-y|),\qquad x,y\in \RR,\: \: t\in\R
$$
for some constant $C$ independent of $t$. Note that the uniform
bound \eqref{linfini_bound} holds true provided $f$  belongs to
$L^p\cap L^q$ for some $p<2$ and $q>2$. However, the
almost-Lipschitz estimate requires the assumption $f \in
L^{\infty}$.

In our situation, we will always deal with $f=\om(t)$ and
$v=g=K\ast_x \om(t)$, where $\om \in \lllll$. For this reason, the
estimates above will actually hold uniformly with respect to time.\\

Finally, we define $\chi_0:\RR \to \R$ to be a smooth, \emph{radial} cut-off map such that
\begin{equation}
\begin{split}
 \label{cut-off}
\chi_0\equiv 0\:\:\textrm{on} \:\: B(0,\frac{1}{2}),\quad
 \chi_0\equiv 1\:\: \textrm{on}\:\:B(0,1)^c ,\quad
0\leq \chi_0\leq 1.
\end{split}
\end{equation}
For a small and positive $\delta$, we set
\begin{equation*}
\chi_\delta(z)=\chi_0\left(\frac{z}{\delta}\right),
\end{equation*}
so that as $\delta$ goes to $0$, we have
\begin{equation}
\label{cut-off2}
\chi_\delta \to 1 \: \textrm{a.e.},\qquad \|\nabla\chi_\delta\|_{L^1(\RR)}\to 0.
\end{equation}
In the sequel, we denote by $u$ the velocity field
$$u\equiv v+H.$$
It is composed of an almost-Lipschitz part $v$ and of a part $H$ which is singular at the point vortex $z(t)$ and smooth outside. Clearly, multiplying any test function by $\chi_\delta \left(x-z(t)\right)$ provides a test function having compact support away from the singularity. This observation will allow us in the subsequent proofs to avoid the singularity and to first perform computations with smooth vector fields. In a second step, we will pass to the limit $\delta \to 0$. This will be readily achieved since we have, thanks to the explicit form of $H$ and the fact that $\chi_\delta$ is radial
\begin{equation}\label{cut.H}
H(t,x)\cdot\na \chi_\delta \left(x-z(t)\right)\equiv 0.
\end{equation}

\textbf{Proof of Theorem \ref{thmequivalence}}. Given $(\om,z,\phi)$
a solution of \eqref{mixte}, it actually suffices to show that
\begin{equation}
\label{formtourb}
\partial_t \om + \diver ((v+H) \om)=0
\end{equation}
in the sense of distributions on $\R^+\times \RR$. \\
We first give a formal proof of \eqref{formtourb}. Let us take a
$C^1$ function $\psi(t,x)$ and define
\begin{equation*}
f(t)=\int_{\RR} \om(t,y) \psi(t,y)\,dy.
\end{equation*}
We set $y=\phi_t(x)$. Since $\phi_t$ preserves Lebesgue's measure
for all time and since $\om$ is constant along the trajectories, we
get
\begin{equation*}
f(t)=\int_{\RR}\om_0(x)\psi (t,\phi_t(x))\,dx.
\end{equation*}
Differentiating with respect to time and using the ODE solved by
$\phi_t(x)$, this leads to
\begin{equation*}
f'(t)=\int_{\RR} \om_0(x) \big( \partial_t \psi +u \cdot \nabla
\psi\big)(t,\phi_t(x))\,dx.
\end{equation*}
Using the change of variables $y=\phi_t(x)$ once more yields
\begin{equation*}
f'(t)=\int_{\RR} \om(t,y) \big( \partial_t \psi+u\cdot \nabla
\psi\big)(t,y)\, dy,
\end{equation*}
which is \eqref{formtourb} in the sense of distributions. In order to justify
the previous computation, we need to be able to differentiate inside
the integral, and we proceed as follows.

Let $\psi=\psi(t,x)$ be any test function. For
$0<\delta<1$, we set
\begin{equation*}
\psi_{\delta}(t,x)\equiv \chi_\delta \left(x-z(t)\right)\,
\psi(t,x),
\end{equation*}
where $\chi_\delta$ is the map defined in \eqref{cut-off}.
Since we have  $
\psi_\delta(t)\equiv 0$ on the ball $B\left(z(t),\frac{\delta}{2}\right)$, we may apply the
previous computation to $\psi_\delta$, which yields for all $t$
\begin{equation}
\label{r1}
\begin{split}
\int_{\RR} \om(t,x) \psi_{\delta}(t,x)\,dx\, - &\int_{\RR} \om_0(x)
\psi_{\delta}(0,x)\,dx
\\& =\int_0^t \int_{\RR} \om \:( \partial_t \psi _\delta +u\cdot \nabla \psi_\delta )\,dx \,ds.
\end{split}
\end{equation}
We first observe that thanks to the pointwise convergence of $\psi_\delta(t,\cdot)$ to $\psi(t,\cdot)$ as $\delta \to 0$, we have
\begin{equation}
\label{r2}
\begin{split}
\int_{\RR} \om(t,x) \psi_{\delta}&(t,x)\,dx\, - \int_{\RR} \om_0(x)
\psi_{\delta}(0,x)\,dx \\
&\to \int_{\RR} \om(t,x) \psi(t,x)\,dx - \int_{\RR} \om_0(x)
\psi(0,x)\,dx
\end{split}
\end{equation}
by Lebesgue's dominated convergence Theorem. Then, we compute
\begin{equation*}
\begin{split}
\dt \psi_{\delta} +u\cdot \nabla
\psi_{\delta}=\chi_{\delta}(x-z)&( \dt \psi + u\cdot \nabla
\psi) \\ &+ \psi (-\dot{z}+v+H) \cdot \nabla
\chi_{\delta}(x-z).
\end{split}
\end{equation*}
Using \eqref{cut.H} and that $v$ is
uniformly bounded, we obtain
\begin{equation}
\begin{split}
\big|\int_{\RR}\om[\dt \psi_{\delta} +&u\cdot \nabla
\p_{\delta}-\chi_{\delta}(x-z)\big( \dt \psi + u\cdot \nabla
\psi\big)]\,dx \big|\\ &\leq C \|\psi \|_{L^{\infty}} \|v
\|_{L^{\infty}}\|\om \|_{L^{\infty}}\int_{\RR} |\nabla
\chi_{\delta}|(x)\,dx.
\end{split}
\label{r3}
\end{equation}
We now let $\delta$ tend to zero. Since $H$ is locally integrable, we observe that
\begin{equation*}
\begin{split}
\int_0^t\int_{\RR}\om\chi_{\delta}\big( \dt \psi + &u\cdot \nabla
\psi\big)\,dx\,ds\\
&\to \int_0^t \int_{\RR} \om ( \dt \psi + u\cdot \nabla
\psi)\,dx\,ds,
\end{split}
\end{equation*}
 so that the conclusion finally follows from  \eqref{cut-off2}, \eqref{r1}, \eqref{r2} and
\eqref{r3}.

\section{Uniqueness of Eulerian solutions}

\label{section3}

This section is devoted to the proof of Theorem \ref{theo}. A first
step in this direction consists in proving that if the vorticity is
initially constant near the point vortex, this remains true for all
time. This is proved in \cite{mar_pul} for any Lagrangian solution by estimating the
distance between the flow and the point vortex. In the present situation where Eulerian solutions are considered, this is achieved by proving that the vorticity of an
Eulerian solution is a renormalized solution in the sense of DiPerna
and Lions \cite{dip-li} of its transport equation.

We recall from Lemmas \ref{calderon} and \ref{loglip} that if
$(\om,z)$ is an Eulerian solution of \eqref{t1}, then the velocity
field defined by $v=K*\om$ satisfies
\begin{equation*}
v\in L^{\infty}(\R^+ \times \RR)\cap L^{\infty}\left(\R^+,
W_{\loc}^{1,1}(\RR)\right)\cap L^\infty(\mathcal{AL}).
\end{equation*}

\subsection{Renormalized solutions}

\label{subsection31}

In what follows, we consider equation \eqref{t1} as a linear
transport equation with given velocity field $u=v+H$ and trajectory
$z$. Our purpose here is to show that if $\om$ solves this linear
equation, then so does $\beta(\om)$ for a suitable smooth function
$\b$. When there is no point vortex, this directly follows from the
theory developed in \cite{dip-li} (see also \cite{dej} for more
details). The results stated in \cite{dip-li} hold for velocity
fields having enough Sobolev regularity; a typical relevant space is
$L_{\loc}^1\left(\R^+,W_{\loc}^{1,1}(\RR)\right)$. These results can
actually be extended to our present situation, thanks to the
regularity of $H$ away from the point vortex and to its special
form.

We define
\begin{equation*}
\Sigma=\{\big(t,z(t)\big),\:t\in \R^+\}
\end{equation*}
and denote by $\mathcal{G}$ its complement in $\R^+\times \RR$. \\
 The starting point in \cite{dip-li} and \cite{dej} is to look at
mollifiers
\begin{equation*}
\om_\eps=\rho_\eps\ast_x \om,\qquad \om_{\eps,\eta}=\om_\eps\ast_t
\theta_{\eta},
\end{equation*}
where $\rho_\eps$ and $\theta_\eta$ are standard regularizing
kernels on $\RR$ and $\R^+$ respectively. We also set, for $f\in
L_{\loc}^1(\R^+\times \RR)$, $f_\eta=f\ast_x\rho_\eta\ast_t
\theta_\eta$.  The following Lemma is a direct consequence of the Sobolev regularity of $v$ and the regularity of $H$ in $\mathcal{G}$.

\newcommand{\oo}{\om_{\eps,\eta}}
\newcommand{\rr}{r_{\eps,\eta}}
\newcommand{\aaa}{a_{\eps,\eta}}
\newcommand{\bbb}{b_{\eps,\eta}}
\newcommand{\ccc}{c_{\eps,\eta}}

\begin{lemma}[Commutators]
\label{commutators}
 Let $(\om,z)$ be an Eulerian solution of \eqref{t1}. Then we have
\begin{equation}
\label{smooth}
 \dt \oo + u_\eta \cdot \nabla \oo=\rr
\end{equation}
in the sense of distributions, where the remainder $\rr$ is defined
by
\begin{equation*}
\rr=u_\eta \cdot \nabla \oo-(u\cdot \nabla \om)_{\eps,\eta}
\end{equation*}
and satisfies
\begin{equation*}
\lim_{\eps \to 0} \left( \lim_{\eta\to 0} \rr \right) = 0 \qquad
\textrm{in }\: L_{\loc}^1(\mathcal{G}).
\end{equation*}
\end{lemma}

\begin{proof}
Let $K$ be a
compact subset of $\mathcal{G}$. Then
there exists $a>0$ such that
\begin{equation*}
|x-z(\tau)|\geq a,\qquad \forall (\tau,x)\in K.
\end{equation*}
We set $\chi(t,x)=\chi_0\left(\frac{x-z(t)}{a/2}\right)$,
where $\chi_0$ is defined in \eqref{cut-off}.
Clearly, we have for $\eta$ and $\eps$ sufficiently small with
respect to $a$ and for $(x,\tau)\in K$
\begin{equation*}
\rr(\tau,x)=(u\chi)_\eta  \cdot \nabla \oo-((u\chi)\cdot
\nabla\om)_{\eps,\eta}.
\end{equation*}
Firstly, the velocity $v$, and hence $v\chi$ belongs to  $L^\infty_{\loc}\left(\R^+,W_{\loc}^{1,1}(\RR)\right)$.
Secondly, thanks to the equality
\begin{equation*}
|H(\tau,x)-H(\tau,y)|=\frac{|x-y|}{2\pi |x-z(\tau)||y-z(\tau)|},
\end{equation*}
we infer that $ H \chi \in
L^{\infty}_{\loc}\left(\R^+,W_{\loc}^{1,1}(\RR)\right)$. Invoking Lemma II.1 in \cite{dip-li} or Lemma 1 in \cite{dej}, we obtain
\begin{equation*}
\lim_{\eps \to 0}\left( \lim_{\eta \to 0} \rr\right)=0\qquad
\textrm{in }\: L_{\loc}^1(K).
\end{equation*}
The Lemma is proved.
\end{proof}

The second step is to use the explicit form of $H$.

\begin{lemma}
 \label{renorm1}
Let $(\om,z)$ be a solution of \eqref{t1}. Let $\beta:\R \rightarrow
\R$ be a smooth function such that
\begin{equation*}
|\beta'(t)|\leq C(1+ |t|^p),\qquad \forall t\in \R,
\end{equation*}
for some $p\geq 0$. Then for all test function $\psi \in
\mathcal{D}(\R^+ \times \RR)$, we have
\begin{equation*}
\frac{d}{dt}\int_{\RR} \psi \beta(\om)\,dx=\int_{\RR} \beta(\om)
(\dt \psi +u\cdot \nabla \psi)\,dx \:\: \textrm{in }\:
L_{\loc}^1(\R^+).
\end{equation*}
\end{lemma}

\begin{proof}  We consider the mollifier $\oo$
defined above. Since $\oo$ is smooth, equality \eqref{smooth}
actually holds almost everywhere in $\R^+\times\RR$. Multiplying
\eqref{smooth} by $\beta'(\oo)$ yields
\begin{equation}
\label{e1} \dt \beta(\oo)+ u_\eta\cdot \nabla \beta(\oo)=\beta'(\oo)
\rr \qquad \textrm{a. e. in }\: \RR.
\end{equation}
We proceed now as in the proof of Theorem \ref{thmequivalence}. Let
$\psi \in \mathcal{D}(\R^+\times \RR)$ be any test function. Let
$\delta>0$ denote a small parameter and let $\chi_\delta$ be the
smooth radial map on $\RR$ defined in \eqref{cut-off}.

We next set
\begin{equation*}
\psi_{\delta}(t,x)=\chi_{\delta}\left(x-z(t)\right) \psi(t,x),
\end{equation*}
and
\begin{equation*}
\psi_{\delta,n}(t,x)=\chi_{\delta}\left( x-z^n(t)\right)\psi(t,x),
\end{equation*}
where $z^n(t)$ is a smooth approximation of $z(t)$. The functions $\psi_\delta$ and $\psi_{\delta,n}$ are supported in $\mathcal{G}$. Note that since
$z$ is Lipschitz with Lipschitz constant given by
$\|v\|_{L^{\infty}(\R^+\times \RR)}$, we may choose $z^n(t)$ so that
\begin{equation}
\label{e5}
 \sup_{t\in\R^+} |z(t)-z^n(t)|\leq \frac{1}{n} \|v\|_{L^{\infty}},\quad
\sup_{ t\in \R^+}|\dot{z}^n(t)|\leq \|v\|_{L^{\infty}}.
\end{equation}

Multiplying \eqref{e1} by
$\psi_{\delta,n}$ and integrating in space gives for all $t$
\begin{equation*}
\begin{split}
\frac{d}{dt} \int_{\RR} \psi_{\delta,n}(t)&\beta(\oo(t))\,dx=\int_{\RR} \beta'(\oo)\rr \psi_{\delta,n}\,dx\\
&+\int_{\RR} \beta(\oo)(\dt \psi_{\delta,n} +u_\eta\cdot \nabla
\psi_{\delta,n})\,dx.
\end{split}
\end{equation*}
Given $\delta>0$, we find $n$ sufficiently large so that $\frac{1}{n}\|v\|_{L^\infty}$ is small with respect to $\delta$.
We infer from the definition of $\chi_\delta$ and \eqref{e5}
that $\psi_{\delta,n}$ is compactly supported in $\mathcal{G}$. We
may thus invoke Lemma \ref{commutators}, the assumption on $\b$ and uniform $L^{\infty}$ bounds
for $\oo$ to deduce that for fixed $\delta$ and $n$
\begin{equation}
\label{e6} \lim_{\eps \to 0}\left(\lim_{\eta \to 0} \int_{\RR}
\beta'(\oo) \rr \psi_{\delta,n}\,dx \right)= 0 \qquad \textrm{in }\:
L_{\loc}^1(\R^+).
\end{equation}
Besides,
\begin{equation*}
\lim_{\eps \to 0}\left(\lim_{\eta \to 0}
\|\oo-\om\|_{L_{\loc}^1\left(\R^+,L^1(\RR)\right)}\right)=0,
\end{equation*}
so that using the uniform bounds on $\dt \psi_{\delta,n}
+u_\eta\cdot \nabla \psi_{\delta,n}$ with respect to $\eta,\eps$, we
are led to
\begin{equation}
\label{e7}
\begin{split}
&\lim_{\eps \to 0}\left(\lim_{\eta \to 0} \int_{\RR} \beta(\oo)(\dt
\psi_{\delta,n} +u_\eta\cdot
\nabla \psi_{\delta,n})\,dx \, d \tau\right)\\
&=\int_{\RR} \beta(\om)(\dt \psi_{\delta,n} +u\cdot \nabla
\psi_{\delta,n})\,dx \, d\tau \qquad \textrm{in }\:
L_{\loc}^1(\R^+).
\end{split}
\end{equation}
Finally, since
\begin{equation*}
\lim_{\eta,\eps \to 0}\frac{d}{dt} \int_{\RR} \beta(\oo)
\psi_{\delta,n}\,dx=\frac{d}{dt} \int_{\RR} \beta(\om)
\psi_{\delta,n}\,dx
\end{equation*}
in the sense of distributions on $\R^+$, we infer from \eqref{e6}
and \eqref{e7}
\begin{equation}
\label{e3}
 \frac{d}{dt} \int_{\RR} \beta(\om) \psi_{\delta,n}\,dx=\int_{\RR} \beta(\om)(\dt \psi_{\delta,n} +u\cdot
\nabla \psi_{\delta,n})\,dx  \qquad \textrm{in }\: \mathcal{D}'(\R^+).
\end{equation}
On the other hand, we compute
\begin{equation*}
\begin{split}
\dt \psi_{\delta,n} +u\cdot \nabla
\psi_{\delta,n}=\chi_{\delta}(x-z^n)&( \dt \psi + u\cdot
\nabla \psi) \\ &+ \psi (v-\dot{z^n}+H) \cdot \nabla
\chi_{\delta}(x-z^n).
\end{split}
\end{equation*}
We first let $n$ go to $+\infty$. Since $\chi_\delta$ is radially
symmetric, we have
\begin{equation*}
H\cdot \nabla \chi_\delta (x-z^n)\to H\cdot \nabla \chi_\delta
(x-z)\equiv 0 \qquad \textrm{in }\: L_{\loc}^1(\R^+\times \RR).
\end{equation*}
Thanks to the pointwise convergence of $\psi_{\delta,n}$ as $n$ goes to
$+\infty$ to $\psi_\delta$ and to the uniform $L^{\infty}$ bounds for
the velocity and the vorticity, we deduce
\begin{equation*}
\begin{split}
\big|\int_{\RR} \beta(\om)(\dt \psi_{\delta} +u\cdot \nabla
\psi_{\delta})\,dx -\int_{\RR} \beta(\om)
&\chi_{\delta}\big(x-z\big) (\dt \psi +u\cdot \nabla
\psi)\,dx\big|\\
& \leq C\int_{\RR} |\nabla \chi_{\delta}|\,dx.
\end{split}
\end{equation*}
Letting $\delta$ go to zero and using \eqref{cut-off2} and \eqref{e3}
yields
\begin{equation*}
\frac{d}{dt} \int_{\RR} \beta(\om)\psi(t,x)\,dx=\int_{\RR}
\beta(\om) (\dt \p +u\cdot \na \psi)\,dx
\end{equation*}
in the sense of distributions on $\R^+$. Since the right-hand side
in the previous equality belongs to $L_{\loc}^1(\R^+)$, the equality
holds in $L_{\loc}^1(\R^+)$ and the Lemma is proved.
\end{proof}

\begin{remark}
\label{remark : conserv} (1) Lemma \ref{renorm1} actually still
holds when $\psi$ is smooth, bounded and has bounded first
derivatives in time and space. In this case, we have to consider
smooth functions $\beta$ which in addition satisfy $\beta(0)=0$, so
that $\beta(\om)$ is integrable.
 This may be proved by approximating $\psi$ by smooth and compactly supported functions $\psi_n$ for which Lemma \ref{renorm1} applies, and by letting then $n$ go to $+\infty$.\\
(2) We let $1\leq p<+\infty$. Approximating $\beta(t)=|t|^p$ by
smooth functions and choosing $\psi\equiv 1$ in Lemma \ref{renorm1},
we deduce that for an Eulerian solution $\om$ to \eqref{t1}, the
maps $t\mapsto \|\om(t)\|_{L^p(\RR)}$ are continuous and constant.
In particular, we have
\begin{equation*}
 \|\om(t)\|_{L^1(\RR)}+\|\om(t)\|_{L^\infty(\RR)}\equiv \|\om_0\|_{L^1(\RR)}+\|\om_0\|_{L^\infty(\RR)},
\end{equation*}
and we denote by $\|\om_0\|$ this last quantity.

\end{remark}

\subsection{Conservation of the vorticity near the point vortex}

\label{subsection32}

Specifying our choice for $\beta$ in Lemma \ref{renorm1}, we are led to the
following

\begin{proposition}
\label{constant_vorticity} Let $(\om,z)$ be an Eulerian solution of
\eqref{t1} such that
\begin{equation*}
\om_0\equiv \a \qquad \textrm{on }\: B(z_0,R_0)
\end{equation*}
for some positive $R_0$. Then there exists a continuous and positive
function $t\mapsto R(t)$ depending only on $t$, $R_0$ and
$||\om_0||$ such that $R(0)=R_0$ and
\begin{eqnarray*}
\forall t\in \R^+, \qquad \om(t)\equiv \a \qquad \textrm{on }\:
B\left(z(t),R(t)\right).
\end{eqnarray*}
\end{proposition}
\begin{proof} We set $\beta(t)=(t-\a)^2$ and use Lemma \ref{renorm1} with this choice. Let $\F \in \mathcal{D}(\R^+\times \RR)$. We claim that for all $T$
\begin{equation*}
\begin{split}
\int_{\RR} \F(T,x) (\om-\a)^2(T,x)\,dx &- \int_{\RR} \F(0,x) (\om-\a)^2(0,x)\,dx \\
&=\int_0^T\int_{\RR} (\om-\a)^2 (\dt \F +u\cdot \nabla \F)\,dx \,
dt.
\end{split}
\end{equation*}
This is actually an improvement of Lemma \ref{renorm1}, in which the
equality holds in $L^1_{\loc}(\R^+)$.  Indeed, we have $\pd_t \om
=-\diver (u\om)$ (in the sense of distributions) with $\om\in
L^\infty$ and $u\in L^\infty(\R^+,L^q_{\loc}(\RR))$ for all $q<2$,
which implies that $\pd_t \om$ is bounded in
$L^1_{\loc}(\R^+,W^{-1,q}_{\loc}(\RR))$. Hence, $\om$ belongs to
$C(\R^+,W^{-1,q}_{\loc}(\RR))\subset C_{w}(\R^+, L^2_{\loc}(\R^2))$, where $C_{w}(L_{\loc}^{2})$ stands for the space of maps $f$ such that for any sequence $t_n\to t$, the sequence $f(t_n)$ converges to $f(t)$ weakly in $L^2_{\loc}$. Since
on the other hand $t \mapsto \|\om(t)\|_{L^2}$ is continuous by
Remark \ref{remark : conserv}, we have $\om \in C(\R^+,L^2(\RR))$. Therefore the previous integral equality holds for all $T$.

Now, we choose a test function $\F$ centered at $z(t)$. More
precisely, we let $\F_0$ be a non-increasing function on $\R$, which
is equal to $1$ for $s\leq 1/2$ and vanishes for $s\geq 1$ and we
set $\F(t,x)=\F_0(|x-z(t)|/R(t))$, with $R(t)$ a smooth, positive
and decreasing function to be determined later on, such that
$R(0)=R_0$. We should regularize $z$ as in the proof of Lemma
\ref{renorm1} to ensure enough regularity for $\F$, but we omit the
details here for the sake of clarity. For this choice of $\F$, we
have $(\om_0(x)-\a)^2 \F(0,x)\equiv 0$.

We compute then
\begin{equation*}
\na \F= \frac{x-z}{|x-z|}\frac{\F_0'}{R(t)}
\end{equation*}
and
\begin{equation*}
\pd_t \F = -\frac{R'(t)}{R^2(t)}|x-z|\F_0'+\frac{\dot{z}\cdot (z-x)}{|x-z|}\frac{\F_0'}{R(t)}.
\end{equation*}
Since $u\cdot \na \F=(v+H)\cdot \na \F=v\cdot \na \F$, we obtain
\begin{equation}\label{ift_calcul}
\begin{split}
\int_{\RR} &\F(T,x) (\om-\a)^2(T,x)\,dx \\
& =\int_0^T \int_{\RR} (\om-\a)^2
\frac{\F_0'(\frac{|x-z|}{R})}{R}\Bigl( (v(x)-\dot{z})\cdot
\frac{(x-z)}{|x-z|}-\frac{R'}{R}|x-z|\Bigl)\, dx\, dt.
\end{split}
\end{equation}
Without loss of generality, we may assume that $R_0\leq 1$, so that $R\leq 1$.
As $\F_0'(\frac{|x-z|}{R})\leq 0$ for $R/2\leq |x-z|\leq R$ and vanishes elsewhere and $R'<0$, we can estimate the right-hand side term of \eqref{ift_calcul} by:
\begin{equation*}
\begin{split}
\int_0^T \int_{\RR} (\om-\a)^2 & \frac{\F_0'(\frac{|x-z|}{R})}{R}\Bigl( (v(x)-\dot{z})\cdot \frac{(x-z)}{|x-z|}-\frac{R'}{R}|x-z|\Bigl)\, dx\, dt   \\
&\leq \int_0^T \int_{\RR} (\om-\a)^2
\frac{|\F_0'|(\frac{|x-z|}{R})}{R}\left(|v(x)-v(z)|+\frac{R'}{2}\right)\,
dx\, dt.
\end{split}
\end{equation*}
Using that $v\in L^\infty(\mathcal{AL})$ and recalling that $\f$ is non-decreasing (see Lemma \ref{loglip}), we deduce from \eqref{ift_calcul}
\begin{equation*}
\begin{split}
\int_{\RR} \F(T,x) &(\om-\a)^2(T,x)\,dx \\
&\leq \int_0^T \int_{\RR} (\om-\a)^2 \frac{|\F_0'|(\frac{|x-z|}{R})}{R}\left(C\f(|x-z|)+\frac{R'}{2}\right)\, dx\, dt\\
&\leq \int_0^T \int_{\RR} (\om-\a)^2
\frac{|\F_0'|(\frac{|x-z|}{R})}{R}\left(CR(1-\ln
R)+\frac{R'}{2}\right)\, dx\, dt,
\end{split}
\end{equation*}
where $C$ only depends on $\|\om_0\|$. Taking $R(t) =\exp (1-(1-\ln
R_0)e^{2Ct})$, we arrive at
\begin{equation*}
\int_{\RR} \F(T,x) (\om-\a)^2(T,x)\,dx\leq 0,
\end{equation*}
which ends the proof.
\end{proof}

Proposition \ref{constant_vorticity} provides the following
\begin{corollary}
\label{cstmil} Let $(\om_1,z_1)$ and $(\om_2,z_2)$ be two Eulerian
solutions to \eqref{t1} starting from $(\om_0,z_0)$. Assume in
addition that
\begin{equation*}
\om_0\equiv \a \qquad \textrm{on }\: B(z_0,R_0).
\end{equation*}
Let $T^\ast>0$ be fixed. Then there exists a time $T_C\leq T^\ast$
depending only on $T^\ast$, $\| \om_0\|$ and $R_0$ such that
\begin{equation*}
\om_1(t)\equiv \om_2(t)=\a \qquad \textrm{on }\: B\big(
z(t),\frac{R(t)}{2}\big),\qquad \forall t\in[0,T_C],
\end{equation*}
where $z(t)$ is the middle point of $[z_1(t),z_2(t)]$. Moreover, we
have
\begin{equation*}
z_1(t),z_2(t)\in B\big(z(t),\frac{R(t)}{8}\big).
\end{equation*}
\end{corollary}
\begin{proof}
Let us define $R_m:=\min_{t\in [0,T^\ast]} R(t)>0$, where $R(t)$ is
given in Proposition \ref{constant_vorticity}. Since
$|z_1(0)-z_2(0)|=0$ and $\|v_1\|_{L^{\infty}},\|v_2\|_{L^{\infty}}
\leq C\|\om_0\|$, we have
\begin{equation*}
|z_1(t)-z_2(t)|\leq 2C\|\om_0\|t \leq \frac{R_m}{4},\qquad \forall
t\in [0,T_C],
\end{equation*}
where $T_C=\min(R_m(8 C\| \om_0\|)^{-1},T^\ast)$. Hence, we get
\begin{equation*}
|z_1(t)-z_2(t)|\leq \frac{R(t)}{4},\qquad \forall  t\in [0,T_C],
\end{equation*}
 and this yields
\begin{equation*}
B\big(z(t),\frac{R(t)}{2}\big)\subset B\left(z_1(t),R(t)\right)\cap
B\left(z_2(t),R(t)\right).
\end{equation*}
The conclusion follows from Proposition \ref{constant_vorticity}.
\end{proof}

\begin{remark} We assume that $\om_0$ has compact support.
Considering $\beta(t)=t^2$ in Lemma \ref{renorm1} and adapting the
proof of Proposition \ref{constant_vorticity}, we obtain that
$\om(t)$ remains compactly supported and its support grows at most
linearly. Indeed, if we choose $\F(t,x)=1-\F_0(|x-z(t)|/R(t))$, with
$R(t)$ a smooth, positive and increasing function such that
$R(0)=R_1$, where $\supp \om_0 \subset B(z_0,R_1)$, then
\eqref{ift_calcul} becomes
\begin{equation*}
\begin{split}
\int_{\RR} &\F(T,x) \om^2(T,x)\,dx \\
& =\int_0^T \int_{\RR} \om^2 \frac{-\F_0'(\frac{|x-z|}{R})}{R}\Bigl( (v(x)-\dot{z})\cdot \frac{(x-z)}{|x-z|}-\frac{R'}{R}|x-z|\Bigl)\, dx\, dt\\
&\leq \int_0^T \int_{\RR} \om^2
\frac{|\F_0'|(\frac{|x-z|}{R})}{R}\Bigl(2C -\frac{R'}{2}\Bigl)\,
dx\, dt,
\end{split}
\end{equation*}
where $C$ depends only on $\|\om_0\|$. The right-hand side is
identically zero for $R(t)=R_1 + 4Ct$, and we conclude that
$\supp(\om(t))\in B(0,R(t))$. \label{remark : support}
\end{remark}


\subsection{Weak formulation for the velocity}

We now turn to the equation satisfied by the velocity $v$ for an
Eulerian solution $(\om,z)$ of \eqref{t1}. This equation is
established in \cite{ift_lop} in the situation where the point
vortex is fixed at the origin. It can be easily extended to our
case, and we obtain the following

\begin{proposition}
\label{form_velocity} Let  $(\om,z)$ be a global solution to
\eqref{t1} with initial condition $(\om_0,z_0)$. Then we have in the
sense of distributions on $\R^+ \times \RR$ 
\begin{equation*}
\begin{cases}
\pd_t v +v\cdot \na v + \diver(v\otimes H+H\otimes v) - v(z(t))^\perp  \delta_{z(t)} = -\na p \\
\diver v =0 \\
\dot{z}(t)= v(t,z(t))\\
v(x,0)=K\ast \om_0 \text{\ and } z(0)=z_0,
\end{cases}
\end{equation*}
where $\delta_{z(t)}$ is the Dirac mass centered at $z(t)$ and
$H(t,x)\equiv K(x-z(t))$.
\end{proposition}

In the sequel,  we will denote by $W^{1,4}_\s(\RR)$ the set of functions belonging to $W^{1,4}(\RR)$ and which are divergence-free in the sense of distributions, and by  $W^{-1,4/3}_\s(\RR)$
its dual space.

 Given two solutions $(\om_1,z_1)$ and $(\om_2,z_2)$ of \eqref{t1}, we define $\vti =K\ast (\om_1-\om_2)=v_1-v_2$. As a consequence of Proposition \ref{form_velocity}, we obtain the following properties for $\vti$.

\begin{proposition}
\label{prop : cont-velocity} Let $\om_0\in \llll$ be compactly
supported, $z_0\in \RR$ and $(\om_1,z_1)$, $(\om_2,z_2)$ be two
Eulerian solutions of \eqref{t1} with initial condition
$(\om_0,z_0)$. Let $\vti=v_1-v_2$. Then we have
\begin{equation*}
\vti \in L_{\loc}^2\left(\R^+,W^{1,4}_\s(\RR)\right),\quad \pd_t \vti \in L_{\loc}^2\left(\R^+,W^{-1,\frac{4}{3}}_\s(\RR)\right).
\end{equation*}
In addition, we have $\vti\in  C\left(\R^+, L^2(\RR)\right)$ and for all $T\in\R^+$,
\begin{equation*}
\|\tilde v(T)\|_{L^2(\RR)}^2=2\int_0^T \langle \pd_t \tilde v,\tilde
v \rangle_{W^{-1,4/3}_\s,W^{1,4}_\s}\,ds,\qquad \forall T\in \R^+.
\end{equation*}
\end{proposition}

\begin{proof}

We define $\tilde{\om}=\om_1-\om_2$, so that $\tilde{v}=K\ast
\tilde{\om}$ and we have for all $t$
\begin{equation*}
\int_{\RR} \om_1(t,x)\,dx\equiv \int \om_0(x)\,dx\equiv \int_{\RR}
\om_2(t,x)\,dx.
\end{equation*}
To see this, we may for instance choose $\beta(t)\equiv t$ and $\p
\equiv 1$ in Lemma \ref{renorm1}. Therefore, $\int
\tilde{\om}(t)\equiv 0$. On the other hand, $\om_1$ and $\om_2$ are
compactly supported in view of Remark \ref{remark : support}, so we
first infer that $\vti (t) \in L^2(\RR)$ for all $t$ (see
\cite{maj-bert} for more details). Using that
$\|\om_i\|_{L^1(\RR)\cap L^\infty(\RR)} \in L^\infty(\R^+)$, we even
obtain
\begin{equation}
\label{bound-velocity}
 \vti \in L_{\loc}^{\infty}(\R^+,L^2(\RR)).
\end{equation}

We now turn to the first assertion in Proposition \ref{prop : cont-velocity}. We apply Proposition
\ref{form_velocity} to $(v_i,z_i)$ for $i=1$ and $i=2$.
 First, we infer from Lemmas \ref{calderon} and \ref{loglip} that  $v_i=K*\om_i$ belongs to $L^\infty(\R^+\times \RR)$ and its
gradient $\na v_i$ to $L^\infty(\R^+,L^4(\RR))$. On the other hand,
since the vorticity $\om_i$ is compactly supported, we have for
large $|x|$
\begin{equation*}
|v_i(t,x)|\leq \frac{C}{|x|} \int_{\RR} |\om_i(t,y)|\, dy,
\end{equation*}
hence $v_i$ belongs to $L_{\loc}^{\infty}(\R^+,L^p(\RR))$ for all $p>2$.
It follows in particular that
\begin{equation*}
v_i \in L_{\loc}^{\infty}(\R^+,W^{1,4}(\RR))
\end{equation*}
and also that $v_i \otimes v_i$ belongs to $L_{\loc}^\infty(L^{4/3})$. Since $v_i$ is divergence-free, we have $v_i \cdot \na v_i=\diver(v_i \otimes v_i)$, and so
$
v_i \cdot \na v_i \in
L_{\loc}^2\big(\R^+,W^{-1,\frac{4}{3}}(\RR)\big).
$
 Moreover,  $v_i(t)\otimes H_i(t)$ belongs to $L_{\loc}^{4/3}$,
 whereas at infinity, $H_i$ and $v_i$ are bounded
by $C/|x|$ which belongs to $L^{8/3}$. This yields
\begin{equation*}
\diver(v_i\otimes H_i),\:\:\diver(H_i\otimes v_i)\in
L_{\loc}^2\big(\R^+,W^{-1,\frac{4}{3}}(\RR)\big).
\end{equation*}
 Besides, we deduce from the embedding of
$W^{1,4}(\RR)$ in $C_0^0(\RR)$ that $\delta_{z_i}$ belongs to
$L_{\loc}^2(W^{-1,\frac{4}{3}})$. Therefore, $v_i \delta_{z_i}\in
L_{\loc}^2(\R^+,W^{-1,\frac{4}{3}}(\RR)).$

According to Proposition \ref{form_velocity}, we finally obtain
$$\langle \pd_t  v_i ,\F\rangle =\langle \pd_t v_i -\na p_i,\F\rangle
\leq C\|\F\|_{L^2(W^{1,4}_\s)}$$
for all divergence-free smooth vector field  $\F$. This implies that
\begin{equation*}
\pd_t v_i \in L_{\loc}^2\big(\R^+,W^{-1,4/3}_\s(\RR)\big), \quad i=1,2,
\end{equation*}
and the same holds for $\pd_t \vti$. Now, since $\tilde v$ belongs to
$L^2_{\loc}\big(\R^+,W^{1,4}_\s\big)$, we deduce from \eqref{bound-velocity} and Lemma 1.2 in Chapter III of
\cite{temam} that $\tilde v$ is almost everywhere
equal to a function continuous from $\R^+$ into $L^2$ and we have
in the sense of distributions on
$\R^+$:
\begin{equation*}
 \frac{d}{dt} \|\tilde v\|_{L^2(\RR)}^2 = 2
\langle \pd_t \tilde v,\tilde v \rangle_{W^{-1,4/3}_\s,W^{1,4}_\s}.
\end{equation*}
We finally conclude by using the fact that $\vti(0)=0$.
\end{proof}

\subsection{Proof of Theorem \ref{theo}}\label{sub34}

In this paragraph, we provide the proof of Theorem \ref{theo} by
making use of the equation for the velocity. To that aim, we let
$(\om_i,z_i)$, $i=1,2$ be two Eulerian solutions of \eqref{t1}, and
we follow the same notations as in the previous paragraph. From
Proposition \ref{prop : cont-velocity}, we may introduce
$$
r(t)\equiv \|
v_1(t,\cdot)-v_2(t,\cdot)\|^2_{L^2(\R^2)}+|z_1(t)-z_2(t)|^2.
$$
Let us fix a positive time $T^\ast$. We will show that $r$ is
identically zero on $[0,T^\ast]$ by mean of a Gronwall type
argument. Since $T^\ast$ is arbitrary, this will provide uniqueness
on the whole of $\R^+$. Let
\begin{equation*}
R_m=\min_{t\in [0,T^\ast]} R(t)=R(T^*), \quad R_M=\max_{t\in [0,T^\ast]} R(t)=R(0)
\end{equation*}
where $R(t)$ is the function defined in Proposition
\ref{constant_vorticity} and let $T_C$ be the time introduced in
Corollary \ref{cstmil}. Since $r(0)=0$, there exists $0<T_0\leq T_C$
such that
\begin{equation*}
r(t)\leq 1,\qquad \forall t\leq T_0.
\end{equation*}

First of all, we take advantage of the fact that $\om_i$ is constant
around the point vortex to state harmonic regularity estimates on
$\tilde v(t)$ in a neighborhood of $z_1(t), z_2(t)$. We recall that
$z(t)$ is the middle point of $[z_1(t),z_2(t)]$.
\begin{lemma} \label{harm} For all $t\leq T_C$, $\tilde v(.,t)$ is harmonic on $B(z(t), R(t)/2)$, with $R(t)>0$ and $T_C$ given in Corollary \ref{constant_vorticity}. In particular, we have the following estimates:
\begin{itemize}
\item[(1)] $\|\tilde v(t,.)\|_{L^\infty(B(z(t), R(t)/4))}\leq C \|\tilde v(t,.)\|_{L^2}$,
\item[(2)] $ \|\na\tilde v(t,.)\|_{L^\infty(B(z(t), R(t)/4))}\leq C \|\tilde v(t,.)\|_{L^2}$,
\end{itemize}
where $C$ only depends on $R(t)$.
\end{lemma}

\begin{proof} In this proof we set $R=R(t)$. In view of Corollary \ref{cstmil} we have  $\curl v_i =\a$ on $B(z(t),R/2)$,
 then $\curl \tilde v=\diver\tilde v=0$ which means that $\tilde v$ is harmonic on this ball: $\Delta \tilde v=0$.
Next, we apply the mean-value Theorem to $\nabla \tilde{v}$ (see e.g. Chapter 2.1 in \cite{g-t}) for all $x\in B=B(z(t),R/4)$:
\begin{equation*}
\na \vti (x) = \frac{1}{\pi (R/8)^2} \int_{B(x,R/8)} \na \vti (y)dy =  \frac{1}{\pi (R/8)^2} \int_{\pd B(x,R/8)} \vti \n ds,
\end{equation*}
therefore
\begin{equation*}
|\na \vti (x)| \leq  \frac{16}{R}  \|\vti \|_{L^\infty(\pd B(x,R/8))}.
\end{equation*}

Now, writing again the mean-value formula for $\tilde{v}$ and $x\in B(z(t),3R/8)$ we obtain
\begin{eqnarray*}
|\vti (x)| &=&\Big| \frac{1}{\pi (R/8)^2} \int_{B(x,R/8)} \vti (y)dy\Big| \\
 &\leq& \frac{1}{\pi (R/8)^2} \int_{B(x,R/8)} |\vti (y)|dy \\
&\leq& \frac{1}{\pi (R/8)^2} \|\vti \|_{L^2}\|1\|_{L^2(B(x,R/8))} \\
&\leq& \frac{8}{\sqrt{\pi}R}\|\vti \|_{L^2}.
\end{eqnarray*}
The conclusion follows.
\end{proof}

The Gronwall estimate for $r(t)$ reads as follows.
\begin{proposition}
\label{gronwall} For all $T\in[0,T_0]$, for all $p\geq 2$, we have
\begin{equation*}
r(T)\leq C\int_0^T
\left[\varphi(r(t))+p\,r(t)^{1-\frac{1}{p}}\right]\,dt,
\end{equation*}
where $C$ depends only on $T^\ast$ and $\| \om_0\|$, and with the
function $\f$ defined in Lemma \ref{loglip}.
\end{proposition}

\begin{proof} We proceed in several steps. Throughout the proof, $C$ will stand for a constant depending only on $R_m$ and $R_M$, therefore on $\|\om_0\|$ and $T^\ast$.\\

\noindent \textbf{Step 1}. We have for the velocities
\begin{equation}\label{step1}
\| \tilde v (T,.)\|_{L^2}^2 \leq C \int_0^T \left[r(t)+\sqrt{r(t)}\f(\sqrt{r(t)})+p\,r(t)^{1-1/p}\right] \,dt,
\end{equation}
$\forall p\geq 2,\qquad \forall T\leq T_C$.

Indeed, subtracting the two equations given by Proposition \ref{form_velocity} for $(v_i,z_i)$, we find
\begin{equation}
\label{diff_velocity}
\begin{split}
\pd_t \tilde v + \tilde v\cdot \na v_1&+v_2\cdot \na\tilde v +
\diver (\tilde v \otimes H_1+ v_2\otimes \tilde H+H_1\otimes \tilde
v+\tilde H\otimes v_2)\\ &- (v_1(z_1)^\perp \cdot
\delta(z_1)-v_2(z_2)^\perp\cdot\delta(z_2)) =-\na \tilde p.
\end{split}
\end{equation}
We then consider smooth and divergence-free functions $\F_n\in C^\infty_c\left( \R^+\times \R^2\right)$  converging to
$\tilde{v}$ in $L^2_{\loc}\left(\R^+,W^{1,4}(\RR)\right)$ as
test functions in \eqref{diff_velocity}, and let $n$ goes to $+\infty$. First, we have for all $T\in \R^+$
\begin{equation*}
\int_0^T \langle \pd_t \tilde v,\F_n
\rangle_{W^{-1,4/3}_\s,W^{1,4}_\s}\, ds\to \int_0^T
\langle \pd_t \tilde v,\tilde{v}
\rangle_{W^{-1,4/3}_\s,W^{1,4}_\s}\,ds,
\end{equation*}
and we deduce the limit in the other terms from the several bounds for $v_i$ stated in the proof of Proposition \ref{prop : cont-velocity}. This yields
\begin{equation}
\label{diff}
\begin{split}
\frac{1}{2}\| \tilde v (T,.)\|_{L^2}^2 =I+J+K,
\end{split}
\end{equation}
where
\begin{equation*}
 \begin{split}
I&=- \int_0^T \int_{\R^2}
\tilde v \cdot(\tilde v\cdot \na v_1+v_2\cdot \na\tilde v)\,dx\,dt,\\
J&= \int_0^T \int_{\R^2} (\tilde v \otimes H_1+ v_2\otimes \tilde H+H_1\otimes \tilde v+\tilde H\otimes v_2):\na\tilde v\,dx\,dt, \\
K&= \int_0^T (v_1(z_1)^\perp \cdot \tilde
v(z_1)-v_2(z_2)^\perp\cdot\tilde v(z_2))\,dt.
\end{split}
\end{equation*}
The next step is to estimate all the terms in the right-hand side.
We now consider times $T\leq T_0$ in (\ref{diff}). In order to
simplify the notation, we set $B=B(z(t),R(t)/4)$.

For the first term $I$ in \eqref{diff}, we begin by noticing that
\begin{equation*}\int_{\RR}(v_2\cdot \na\tilde v)\cdot \tilde v\,dx=\frac{1}{2}\int_{\RR} v_{2}\cdot\na |\tilde v|^2\,dx=-\frac{1}{2}\int_{\RR} |\tilde v|^2 \diver v_2\,dx=0.
\end{equation*}
Moreover, H\"older's inequality gives
\begin{equation*}
\left|\int_{\RR}(\tilde v\cdot \na v_1)\cdot \tilde v\,dx\right|\leq\|\tilde v\|_{L^2}\|\tilde v\|_{L^q}\|\na v_1\|_{L^p},
\end{equation*}
with $\frac{1}{p}+\frac{1}{q}=\frac{1}{2}$. On the one hand, Lemma
\ref{calderon} states that $\|\na v_1\|_{L^p}\leq Cp\|\om_1\|_{L^p}$
for $p\geq 2$. On the other hand, we write by interpolation
$\|\tilde v\|_{L^q}\leq \|\tilde v\|^a_{L^2}\|\tilde
v\|^{1-a}_{L^\infty}$ with
$\frac{1}{q}=\frac{a}{2}+\frac{1-a}{\infty}$. We have that
$a=1-\frac{2}{p}$, so we are led to
\begin{equation}
\label{eq : I}
|I| \leq Cp\int_0^T \|\tilde v\|_{L^2}^{2-2/p}\,dt.
\end{equation}

We now estimate $J$. We have
\begin{equation*}
\begin{split}
\int_{\RR} (\tilde v\otimes H_1): \na \tilde v\,dx&=\int_{\RR}\sum_{i,j}\tilde
v_iH_{1,j}\pd_j\tilde v_i\,dx=\frac{1}{2}\sum_i\int_{\RR}\sum_jH_{1,j}\pd_j\tilde
v_i^2\,dx\\&=-\frac{1}{2}\sum_i\int_{\RR} \tilde v_i^2 \diver H_1\,dx=0,
\end{split}
\end{equation*}
 since $H_1$ is divergence-free, and
\begin{equation}
\label{special}
\begin{split}
\Bigl| \int_0^T \int_{\R^2} (H_1\otimes \tilde v)&:\na\tilde
v\,dx\,dt \Bigl|\ \
 \leq \ \  \Bigl| \int_0^T \int_{B} (H_1\otimes \tilde v ):\na\tilde v\,dx\,dt \Bigl|\\
&+ \Bigl| \int_0^T \int_{B^c} (H_1\otimes \tilde v ):\na\tilde
v\,dx\,dt \Bigl|.
\end{split}
\end{equation}
We perform an integration by part for the second term in the
right-hand side of \eqref{special}. Arguing that $\diver \vti=0$, we obtain
\begin{equation*}
\begin{split}
\Bigl| \int_0^T \int_{\R^2} (H_1\otimes \tilde v):\na\tilde
v\,dx\,dt \Bigl|\ \
\leq \ \ & \Bigl| \int_0^T \int_{B} (H_1\otimes \tilde v ):\na\tilde v\,dx\,dt \Bigl|\\
 + \Bigl|- \int_0^T \Bigl(\int_{B^c} (\tilde v &\cdot  \na H_1 )\cdot \tilde v\,dx +
 \int_{\pd B} (H_1\cdot \vti)(\vti\cdot \n)  ds \Bigl)dt\Bigl| \\
\leq\ \ & \int_0^T  \|H_1\|_{L^1(B)}\|\tilde v\|_{L^\infty(B)}\|\na \tilde v\|_{L^\infty(B)}\,dt \\
 &+ \int_0^T \|\na H_1\|_{L^\infty(B^c)}\|\tilde v\|^2_{L^2} \,dt\\
& + \int_0^T \| H_1\|_{L^\infty(\pd B)}\|\tilde v\|^2_{L^\infty(\pd
B)}|\pd B|\,dt.
\end{split}
\end{equation*}
According to Lemma \ref{harm}, this gives
\begin{equation*}
\Bigl| \int_0^T \int_{\RR} (H_1 \otimes \tilde v):\na \tilde
v\,dx\,dt \Bigl| \ \ \leq\ \  C \int_0^T \|\tilde v\|^2_{L^2}
\,dt.
\end{equation*}
 In the same way, we obtain by integration by part
\begin{equation*}
\begin{split}
\Bigl| \int_0^T \int_{\R^2} (v_2\otimes \tilde H):\na\tilde
v\,dx\,dt \Bigl| \ \
 \leq \ \ &  \Bigl| \int_0^T \int_{B} (v_2\otimes \tilde H):\na\tilde v\,dx\,dt \Bigl|\\
 + \Bigl|- \int_0^T \Bigl(\int_{B^c} (\tilde H&\cdot  \na v_2 )\cdot \tilde v\,dx
 + \int_{\pd B} (v_2\cdot\vti)(\tilde H \cdot\n) ds \Bigl)\,dt\Bigl|.
\end{split}
\end{equation*}
Therefore,
\begin{equation*}
\begin{split}
\Bigl| \int_0^T \int_{\R^2} (v_2\otimes \tilde H):\na\tilde
v\,dx\,dt \Bigl| \ \
\leq\ \ & \int_0^T  \|\tilde H\|_{L^1(B)}\|v_2\|_{L^\infty}\|\na \tilde v\|_{L^\infty(B)}\,dt \\
&+ \int_0^T \| \tilde H\|_{L^\infty(B^c)}\|\tilde v\|_{L^2}\|\na v_2\|_{L^2} \,dt \\
+  &\int_0^T\| \tilde H\|_{L^\infty(\pd B)}\|\tilde v\|_{L^\infty(\pd B)} \| v_2\|_{L^\infty}|\pd B|\,dt.
\end{split}
\end{equation*}
Using again Calder\'on-Zygmund inequality for $v_2$ and Lemma
\ref{harm}, we get
\begin{equation*}
\begin{split}
\Bigl| \int_0^T \int_{\R^2} (v_2\otimes \tilde H):\na\tilde
v\,dx\,dt \Bigl| \ \ \leq\ \ C \int_0^T  \left(\|\tilde H\|_{L^1(B)} +\| \tilde
H\|_{L^\infty(B^c)}\right)\|\tilde v\|_{L^2} \,dt.
\end{split}
\end{equation*}
A very similar computation yields
\begin{equation*}
\begin{split}
\Bigl| \int_0^T \int_{\R^2} (\tilde H\otimes v_2):&\na\tilde
v\,dx\,dt \Bigl| \
\\ &\leq
C \int_0^T \left(\|\tilde H\|_{L^1(B)}+\| \na \tilde H\|_{L^2(B^c)} +\|
\tilde H\|_{L^\infty(\pd B)}\right) \|\tilde v\|_{L^2}\, dt.
\end{split}
\end{equation*}
We need here some estimates for $\tilde H$. We recall that
$\tilde H$ is defined by $\tilde H = H_1-H_2$, so that
$$\| \tilde H\|_{L^\infty(B^c)}= \sup_{z\in B^c}\frac{|z_1-z_2|}{2\pi|z-z_1||z-z_2|}\leq (8/R_m)^2 \frac{|z_1-z_2|}{2\pi}.$$
On the other hand, it follows from potential theory estimates (see
e.g. \cite{maj-bert}) that
\begin{equation*}
\begin{split}
\int_B |\tilde H(x) |\, dx =C \int_{B} |K(x-z_1)-K(x-z_2)|\,dx \leq C \f(|z_1-z_2|).
\end{split}\end{equation*}
Concerning the $L^2$ norm, we observe that for
$x \in B^c$,
\begin{equation*}
|\na\tilde H(x)|\leq |z_1-z_2| \sup_{[x-z_1,x-z_2]} |D^2 K|\leq
\frac{C}{|x-z|^3}|z_1-z_2|,
\end{equation*}
 which implies that $\|\na \tilde
H\|_{L^2(B^c)}\leq C |z_1-z_2|$. Therefore, we arrive at
\begin{equation*}
|J|\leq 2C \int_0^T \Big( \|\tilde v\|_{L^2}+ |z_1-z_2|+\f(|z_1-z_2|)\Big) \|\tilde v\|_{L^2}\,dt.
\end{equation*}
Since $\f$ is increasing, this implies
\begin{equation}
 \label{eq : J}
|J|\leq C \int_0^T \left[r(t)+\sqrt{r(t)}\f(\sqrt{r(t)})\right] \,dt.
\end{equation}

Finally, we decompose the third term $K$ in \eqref{diff} as follows:
\begin{equation*}
 \begin{split}
v_1(z_1)^\perp \cdot \tilde v(z_1)-v_2(z_2)^\perp&\cdot\tilde v(z_2) = \left(v_1(z_1)^\perp-v_1(z_2)^\perp\right)\cdot \tilde v(z_1)\\
& +\left(v_1(z_2)^\perp-v_2(z_2)^\perp\right)\cdot \tilde v(z_1)\\& +
v_2(z_2)^\perp \cdot \left(\tilde v(z_1)-\tilde v(z_2)\right).
\end{split}
\end{equation*}
Applying Lemma \ref{loglipschitz} to $v_1$, we obtain
\begin{equation*}
\begin{split}
|v_1(z_1)^\perp \cdot \tilde v(z_1)-v_2(z_2)^\perp \cdot\tilde v(z_2)| \leq &C|\tilde v(z_1)| \f(|z_1-z_2|) + |\tilde
v(z_2)||\tilde v(z_1)|\\
&+\|v_2\|_{L^\infty} \|\na\tilde{v}\|_{L^\infty([z_1,z_2])} |z_1-z_2|,
\end{split}
\end{equation*}
so that Lemma \ref{harm} finally yields
\begin{equation}
 \label{eq : K}
|K|\leq C \int_0^T \left[r(t)+\sqrt{r(t)}\f(\sqrt{r(t)})\right]\, dt.
\end{equation}
Estimates \eqref{eq : I}, \eqref{eq : J} and \eqref{eq : K} complete the proof of \eqref{step1}. \\ \\

\noindent \textbf{Step 2}. We have for the points vortex
\begin{equation*}
|z_1(T)-z_2(T)|^2 \leq C\int_0^T \left[r(t)+\sqrt{r(t)}\f(\sqrt{r(t)})\right]\,dt,\qquad \forall T\leq T_C.
\end{equation*}

Indeed, since $z_1$ and $z_2$ are Lipschitz, their derivatives exist for almost every time $t$, and we have at these points
\begin{equation*}
\begin{split}
\frac{d}{dt} |z_1-z_2|^2 &= 2 \langle z_1-z_2, v_1(z_1)-v_2(z_2)\rangle \\
&= 2 \langle z_1-z_2,v_1(z_1)-v_1(z_2)\rangle +2 \langle z_1-z_2,
\tilde v(z_2)\rangle.
\end{split}
\end{equation*}
This yields in view of Lemmas \ref{harm} and \ref{loglip}
\begin{equation*}
\frac{d}{dt} |z_1-z_2|^2\leq C|z_1-z_2| \f(|z_1-z_2|) + |z_1-z_2|
\| \tilde v \|_{L^2},
\end{equation*}
and we conclude by integrating the previous inequality.\\

Finally, we observe that for $z\leq 1$, we have  $$z\f(z) \leq
\f(z^2),\qquad z\leq \varphi(z).$$ Proposition \ref{gronwall} then
directly follows from Steps 1 and 2 and the definition of $T_0$.
\end{proof}

Theorem \ref{theo} is now an easy consequence of the following

\begin{lemma}
\label{est} We have for all $p\geq 1$
\begin{equation*}
 \f(t)\leq
p\,t^{1-\frac{1}{p}}, \qquad \forall t\geq 0,
\end{equation*}
where $\f$ is defined in Lemma \ref{loglip}.
\end{lemma}
Let us assume Lemma \ref{est} for a moment and finish the proof of Theorem \ref{theo}. We deduce from
Proposition \ref{gronwall} that for all $T\leq T_0$,
\begin{equation*}
r(T) \leq C\int_0^T p\,r(t)^{1-\frac{1}{p}}\,dt.
\end{equation*}
Using a Gronwall-like argument, this implies
\begin{equation*}
r(T)\leq (CT)^p,\qquad \forall p\geq 2.
\end{equation*}
Letting $p$ tend to infinity, we conclude
that $r(T)\equiv 0$ for all $T<\min (T_0, 1/C)$. Finally, we consider the maximal interval of $[0,T^\ast]$ on which $r\equiv 0$, which is closed by continuity of $r$. If it is not equal to the whole of $[0,T^\ast]$, we may repeat the proof above, which leads to a contradiction by maximality. Therefore uniqueness holds on $[0,T^\ast]$, and this concludes the proof of Theorem \ref{theo}. \\

\textbf{Proof of Lemma \ref{est}}. The
result is obvious for $t\geq 1$. Let $f_p(t)=t^{\frac{1}{p}} (1-\ln
t)$. It suffices to show that $f_p(t)\leq p$. Computing
$f_p'(t)=t^{\frac{1}{p}-1} \big( \frac{1}{p} (1-\ln t)-1\big)$, we
observe that $f_p'(t)\geq 0$ if and only if $t\leq e^{1-p}$. Then $f_p$ is maximal when $t=e^{1-p}$ and we infer that
\begin{equation*}
f_p(t)\leq f_p(e^{1-p})=pe^{\frac{1}{p}-1}\leq p
\end{equation*}
for $p\geq 1$. This completes the proof.


\section{Final remarks and comments} \label{final}

\subsection{An alternative approach to uniqueness}

In this subsection, we present an alternative approach for proving Theorem \ref{theo}, which was indicated to us by one of the referees. In contrast with our proof, which is uniquely PDE based, it is rather Lagrangian based but still relies on Proposition \ref{constant_vorticity}. Let $T>0$, and let $\om$ be an Eulerian solution of \eqref{t1} on $[0,T]$ satisfying the assumptions of Theorem \ref{theo}. Let $\e$ such that $\e<R(T)$, where $t\mapsto R(t)$ is defined in Proposition \ref{constant_vorticity}. Then $\om$ is also a weak solution of the regularized equation
\begin{equation}
\label{eq : regul_equation}
\pd_{t} \om+\diver\big((v+K_\e(x-z(t))\om\big)=0,
\end{equation}
where $K_\e$ is a smooth, bounded and divergence-free map on $\RR$ which coincides with the Biot-Savart kernel $K$ on $B(0,\e)^c$. Indeed, let $\psi$ be a test function, then for all $t\in [0,T]$ we have
\begin{equation*}
\begin{split}
\int_{\RR} \om K(x-z(t))\cdot\nabla \psi\,dx=&\int_{B(z(t),\e)}\om K(x-z(t))\cdot\nabla \psi\,dx\\
&+\int_{B(z(t),\e)^c}\om K_\e(x-z(t))\cdot \nabla \psi\,dx\\
=&\alpha\int_{B(z(t),\e)}\big[K(x-z(t))-K_\e(x-z(t))\big]\cdot\nabla \psi\,dx\\
&+\int_{\RR}\om K_\e(x-z(t))\cdot\nabla\psi\,dx,
\end{split}
\end{equation*}
where the last equality is due to the fact that $\om(t)\equiv \alpha$ on $B(z(t),\e)$ for $t\in[0,T]$ by Proposition \ref{constant_vorticity}. Using the fact that $K$ and $K_\e$ are divergence-free and integrating by part yields
\begin{equation*}
\int_{B(z(t),\e)}\big[K(x-z(t))-K_\e(x-z(t))\big]\cdot\nabla \psi\,dx=0.
\end{equation*}
Clearly, $v(t)+K_\e(\cdot-z(t))$ is almost Lipschitz for all time, and it can be shown that it is moreover continuous in time. This  means that $\om$ is constant along $C^1$ trajectories $t\mapsto X_\e(x,t)$ satisfying
\begin{equation*}
\frac{d}{dt}X_\e(t,x)=v(t,X_\e(t,x))+K_\e(X_\e(t,x)-z(t)),\quad X_\e(0,x)=x\neq z_0.
\end{equation*}
Therefore, proving uniqueness for the Eulerian formulation on $[0,T]$ when the vorticity is constant near the point vortex amounts to proving uniqueness for this Lagrangian formulation. Since the singular part $H$ is replaced by a bounded and Lipschitz field $H_\e$, this can be achieved by means of Lagrangian methods.

\subsection{Equivalence of Lagrangian and Eulerian formulations}

In Section \ref{section2} of this work, we have proved that Lagrangian solutions are always Eulerian solutions. Given the global existence of Lagrangian solutions proved
by Marchioro and Pulvirenti \cite{mar_pul}, we obtain as a byproduct of Theorem
\ref{theo} that an Eulerian solution is also a Lagrangian solution. Therefore
Definitions \ref{defLAGR} and \ref{defPDE} are equivalent if the
initial vorticity belongs to $\llll$, is compactly supported and
constant near $z_0$.

 As a matter of fact, the renormalization Lemma \ref{renorm1} and Theorem \ref{thmequivalence} enable to establish the equivalence of Lagrangian and Eulerian formulations in the general case, without assuming that $\om_0$ is constant near the point vortex. More precisely, we have

\begin{proposition}
\label{reciproque} Let $(\om,z)$ be an Eulerian solution of \eqref{t1} with
initial datum $(\om_0,z_0)$. Then
\begin{enumerate}
\item $v=K\ast \om\in C(\R^+\times\RR)\cap L^\infty(\R^+,\mathcal{AL})\cap L^\infty(\R^+\times \RR)$,

\item The trajectory $t\mapsto z(t)$ belongs to $C^1(\R^+)$ and
satisfies $\dot{z}(t)=v(t,z(t))$ for all $t$,

\item For all $x\in \RR/\{z_0\}$, there exists a unique and global flow
$\phi_t(x)$ which is $C^1$ in time, and such that
$(\om,v,z,\phi)$ is a global solution to \eqref{mixte}.
\end{enumerate}
\end{proposition}

\begin{proof} 
We start by proving (1), which clearly implies (2). By Lemma \ref{loglip} we only have to show the time continuity. First, we claim that 
\begin{equation}\label{clp}
\om\in C(\R^+,L^p),\quad 1<p<+\infty.
\end{equation}
Indeed, we already know by the proof of Proposition \ref{constant_vorticity} that it holds for $p=2$, and we conclude by interpolation since $\om\in L^\infty(\R^+,L^1\cap L^\infty(\R^2))$.

Next, for all $t,s\in \R^+$, we have
\begin{equation*}
\begin{split}
|v(t,x)-v(s,x)|\leq C\int_{|x-y|\leq 1}
&\frac{|\om(t,y)-\om(s,y)|}{|x-y|}\,dy\\
&+C\int_{|x-y|\geq 1} \frac{|\om(t,y)-\om(s,y)|}{|x-y|}\,dy.
\end{split}
\end{equation*}
We choose $2<q<+\infty$ and $1<p<2$ and apply H\"older's inequality
to each term in the r.h.s.  to get
\begin{equation*}
\sup_{x\in \RR}|v(t,x)-v(s,x)|\leq C \|\om(t)-\om(s)\|_{L^{q}}+ C
\|\om(t)-\om(s)\|_{L^{p}}.
\end{equation*}
Since on the other hand $v\in L^\infty(\mathcal{AL})$, we infer from \eqref{clp} that $v$ is continuous in space and time.

Therefore, since on the other hand $K$ is bounded and Lipschitz away from zero, we may apply the
extension of the Cauchy-Lipschitz Theorem to the class of 
functions satisfying (1) (see Lemma 3.2 in \cite{livrejaune}): for all $x\in \RR/\{z_0\}$, there
exist $T(x)>0$ and a unique $C^1$ trajectory $\phi_t(x)$ on $[0,T(x))$
such that
\begin{equation*}
\frac{d}{dt} \phi_t(x)=v\big(t,\phi_t(x)\big)+K\big(\phi_t(x)-z(t)\big).
\end{equation*}
On the other hand, using the fact that $K(X)\cdot X=0$ for all $X\neq 0$ it can be shown (see \cite{mar_pul}) that for all $t\in [0,T(x))$ we have 
\begin{equation}
 \label{ineq : no-collision}
|\phi_t(x)-z(t)|\geq A(t) |x-z_0|^{B(t)},
\end{equation}
 where $A(t)$ and $B(t)$ are positive functions depending only on $\|\om_0\|$. This implies that $T(x)=+\infty$. 

The fact that $\phi_t$ is an homeomorphism is standard when there is no point vortex. It can be extended to our case to show that $\phi_t$ is an homeomorphism: $\R^2\setminus \{z_0\} \to \R^2\setminus \{z(t)\}$ by using \eqref{ineq : no-collision}.

 Finally, in order to show the conservation of the Lebesgue's measure, we approximate $v$ by a sequence of smooth, divergence-free fields $(v_\eps)_{0<\eps<1}$, and we denote by $\phi_\eps$ and $z_\eps$ the flows associated to $v_\eps+K_\eps(\cdot-z_\eps)$ and $v_\eps$ respectively, where $K_\eps$ is the smooth divergence-free map defined in the previous subsection. Thanks to \eqref{ineq : no-collision}, we readily check that, up to a subsequence, $z_\eps$ converges to $z$ on compact sets of $\R^+$ and $\phi_\eps$ to $\phi$ on compact sets of $\R^+\times \RR\setminus \{z_0\}$. Now, as $v_\eps(t)+K_\eps(\cdot-z_\eps(t))$ is divergence-free, Liouville's Theorem (see e.g. Appendix 1.1 in \cite{livrejaune}) ensures that $\phi_\eps(t)$ preserves Lebesgue's measure for all $t\geq 0$. Letting $\eps$ tend to zero, we thus obtain that it also holds for $\phi_t$. 

It only remains to check that $\om$ is transported by the flow. For that
purpose, we define $\overline{\om}(t,x) =\om_0(\phi_t^{-1}(x))$,
so that $\overline{\om}(0,x)=\om_0(x)$. It follows from Theorem \ref{thmequivalence} that $\overline{\om}\in L^\infty(\R^+,L^1\cap L^\infty(\RR))$ is a weak
solution to the linear transport equation
\begin{equation*}
\dt \overline{\om} + u\cdot \nabla \overline{\om}=0.
\end{equation*}
Now, according to Remark \ref{remark : conserv} (2) applied to $\om-\overline{\om}$, we have
\begin{equation*}
\|\om(t)-\overline{\om}(t)\|_{L^2}\equiv 0,\quad t\geq 0,
\end{equation*}
and we infer that $\om(t,x)=\overline{\om}(t,x)$ for a.e. $x\in \RR$.

This concludes the proof of Proposition \ref{reciproque}.
\end{proof}

\subsection{The case of several point vortices}
In this paper, we have only considered the vortex-wave system with one single point vortex. In the case of a finite number $N$ of vortices $z_i$ with real intensities $d_i,i=1,\ldots,N$, the vortex-wave system \eqref{mixte} modifies as follows:
\begin{equation}\tag{LFN}
\begin{cases}
v(\cdot,t)=(K\ast \omega)(\cdot,t),\\
\displaystyle \dot{z_i}(t)=v(t,z_i(t))+\sum_{\substack{j=1 \\j\neq i}}^N d_j K\left(z_i(t)-z_j(t)\right),\\
 z_i(0)=z_{i,0},\\
\displaystyle \dot{\,\phi}_t(x)=v(t,\phi_t(x))+\sum_{j=1}^N d_j K \left(\phi_t(x)-z_j(t)\right),\\
\phi_0(x)=x, \; x\neq z_{j,0} ,\\
\omega(\phi_t(x),t)=\omega_0(x),\qquad t\in \R.
\end{cases}
\label{mixteN}
\end{equation}
In this situation, every vortex trajectory $z_i(t)$ is submitted to
the fields generated by the other vortices and to the regular field
$v$, and the regular part moves under the action of the field
created by itself and by the $N$ vortices. The velocity fields
appearing in \eqref{mixteN} are well-defined as long as the flow and
the vortices remain separated. If the intensities $d_i$ all have the
same sign, it has been established in \cite{mar_pul} that no
collision among the vortices and the flow can occur in finite time,
and global existence for \eqref{mixteN} has been proved for an
arbitrary initial vorticity $\om_0 \in \llll$ and $N$ distinct
vortices $z_{i,0}$.

In particular, given any time $T>0$, there exists a positive $a$ such that up to time $T$, we have
 $|z_i(t)-z_j(t)|\geq a$. So, the field created by a vortex near the other vortices is Lipschitz and bounded.
  Localizing then the test functions used throughout the proofs in Sections 2 and 3 near each vortex,
  we may extend Theorems \ref{thmequivalence} and \ref{theo} to the case of many vortices and obtain uniqueness
  for the corresponding Eulerian formulation to \eqref{mixteN} when the vorticity is initially constant near each point vortex.
   This gives more precisely
\begin{theorem}
\label{theoN} Let $\om_0\in L^1\cap L^\infty(\RR)$ and
$z_{1,0},\ldots,z_{N,0}$ be $N$ distinct points in $\RR$ with
positive intensities $d_i$. Assume that there exist a positive
$R_0$, which is smaller than the minimal distance between the
initial vortices, and $\a_i\in\R$ such that for all $i$,
\begin{equation*}
\om_0 \equiv \a_i  \: \: \textrm{on }\: B(z_{i,0},R_0).
\end{equation*}
Suppose in addition that  $\om_0$ has compact support. Then there
exists a unique Eulerian solution of the vortex-wave system with
this initial data.
\end{theorem}
It follows in particular from Theorem \ref{theoN} that equivalence
between Eulerian and Lagrangian formulations also holds in the case
of several point vortices.

\subsection{Uniqueness when the vortex point is fixed}

In this subsection, we address the problem of uniqueness to a slightly
different equation from \eqref{t1}. The main difference is
that the point vortex is fixed (for instance at the origin) instead
of moving under the action of the velocity. It reads
\begin{equation}
\label{s1}
\begin{cases}
\dt \om + u\cdot \nabla \om=0, \\
u(t,x)=v(t,x)+\g K(x),\qquad v=K\ast \om,
\end{cases}
\end{equation}
where $\g \in \R$. This system is obtained by Iftimie, Lopes Filho
and Nussenzveig Lopes in \cite{ift_lop} as an asymptotical equation
for the classical Euler equations on exterior domains. More
precisely, they consider a family of obstacles $\OM_\e\equiv \e \OM$
contracting to a point as $\e\to 0$, where $\OM$ is a smooth,
bounded, open, connected, simply connected subset of the plane.
Throughout \cite{ift_lop}, the authors assume that the initial
vorticity $\om_0$ is independent of $\e$, smooth, compactly
supported outside the obstacles $\OM_\e$ and that the circulation
$\g$ of the initial velocity on the boundary is independent of $\e$.
The authors prove that as $\eps$ goes to $0$, the flow converges to
a a global solution of equation \eqref{s1}. Of course, this system
reduces to the classical Euler equations when $\g=0$, for which
uniqueness is known in the class $\lllll$ \cite{yudo}.

Equations \eqref{s1} have also been considered in the Lagrangian
formulation by Marchioro \cite{mar} in the case where the support of
$\om_0$ does not intersect the origin and in a smooth setting. In
this paper, it is proved that for a $C^2$ vorticity, a trajectory
starting away from the origin never reaches it. This provides in
particular uniqueness in the Lagrangian formulation in this case.

According to Section \ref{subsection31} of this work, it is actually
possible to adapt the key idea used in \cite{mar} to equations
\eqref{s1} without relying on the trajectories. In particular, we
first prove that if the initial vorticity $\om_0$ vanishes in a
neighborhood of the origin, then this holds for all time.

\begin{proposition}
\label{constant_vorticity_2} Let $\om$ be a global Eulerian solution
of \eqref{s1} such that
\begin{equation*}
\supp \om_0\subset B(0,R_0^{-1})\setminus B(0,R_0)
\end{equation*}
for some $0<R_0<1$. Then there exist positive constants $C_1$ and
$C_2$ depending only on $R_0$ and $||\om_0||$ such that
\begin{eqnarray*}
\om(t)\equiv 0 \qquad \textrm{on \; \;} B\left(0,C_1e^{-C_2
t}\right),\qquad \forall t\geq 0.
\end{eqnarray*}
\end{proposition}
\begin{proof} We may assume that $\gamma=1$.
As already mentioned, Section \ref{subsection31} of this work
applies equally well to equations \eqref{s1} by replacing the moving
point vortex by the origin. In particular, we infer that $\om^2$ is
also a weak solution of the linear transport equation corresponding
to \eqref{s1}. According to Remark \ref{remark : support}, we have
\begin{equation}
\label{inclusion : support} \supp \om(t) \subset
B\left(0,K(1+t)\right), \qquad \forall t\geq 0,
\end{equation}
where $K$ only depends on the initial conditions $R_0$ and
$||\om_0||$. We aim to apply Lemma \ref{renorm1} with the choice
$\b(t)=t^2$ and we set
\begin{equation*}
\F(t,x)=\chi_0 \left( \frac{-\ln|x|-\int \ln|x-y|\om(t,y)\,dy+C(t)}{2\pi
R(t)}\right),
\end{equation*}
where $\chi_0$ is a smooth function : $\mathbb{R} \to
\R^+$ which is identically zero for $|x|\leq 1/2$ and
identically one for $|x|\geq 1$ and increasing on $\R^+$, $R(t)$ is an increasing
continuous function and $C(t)$ is a continuous function to be determined later on. We set
\begin{equation*}
g(t,x)=\frac{1}{2 \pi}\int_{\RR} \ln|x-y| \om(t,y)\,dy,
\end{equation*}
it follows from \eqref{inclusion : support} that $2\pi|g(t,x)|\leq
C_0\left(1+\ln(1+t)\right)$ for some constant $C_0$. Increasing possibly $C_0$ we also have for $x\in \supp \om(t)$ $-\ln|x|\geq -2\pi C_0 (1+\ln(1+t))$. Therefore, setting $C(t)=2C_0\left( 1+\ln(1+t)\right)$ and
$$
y(t,x)=-\frac{\ln |x|}{2\pi}-g(t,x)+\frac{C(t)}{2\pi}
$$
we see that for all $t$ and $x\in \supp \om(t)$ the term $y(t,x)$ is positive.

 Next, it is proved by
Marchioro in \cite{mar} that if $\om$ is a smooth solution of
\eqref{s1},
\begin{equation*}
\dt g(t,x)=-\int_{\RR}K^{\bot} (y-x)\cdot \left( v(y)+K(y)\right)
\om(t,y)\,dy.
\end{equation*}
Marchioro's paper also states that
\begin{equation}
\label{b1} \|\dt g\|_{L^\infty} \leq C_1,
\end{equation}
where $C_1$ only depends on $\|\om_0\|$ and $R_0$. This can be
extended to weak Eulerian solutions of \eqref{s1} by replacing $\ln$
by $\ln_\eps$ in the definition of $g$, where $\ln_\eps|z|$
coincides with $\ln|z|$ on $B(0,\eps)^{c}$ and is identically equal
to $\ln \eps$ in $B(0,\eps)$. Letting $\eps$ then go to zero, we
deduce that for all $x$, $g(x,\cdot)$ is Lipschitz and has a time
derivative for almost every time; moreover the bound \eqref{b1}
holds at those times. We omit the details here and may consider that
$\F$ is $C^1$.

 On the other hand, we have
\begin{equation*}
\nabla_x g(t,x)=-\int_{\RR} K^{\bot}
(x-y)\om(t,y)\,dy=-v^{\bot}(t,x),
\end{equation*}
therefore
\begin{equation*}
\left(v+K\right)\cdot \nabla \F=\left(v+K\right)\cdot
\left(v^{\bot}+K^{\bot}\right)\frac{\chi_0'}{R}\equiv 0.
\end{equation*}
Besides,
\begin{equation*}
\dt \F(t,x)=\Big(-\frac{R'(t)}{R^2(t)} y(t,x)\Big)
+\frac{1}{R} (-\dt g(t,x)+\frac{C'(t)}{2\pi})\Big)
\chi_0'\left(\frac{y(t,x)}{R(t)}\right).
\end{equation*}

In view of Lemma \ref{renorm1}, this yields
\begin{equation*}
\begin{split}
\int_{\RR} &\F(T,x) \om^2(T,x)\,dx -\int_{\RR} \F(0,x)
\om^2_0(x)\,dx\\& =\int_0^T \int_{\RR} \om^2
\frac{\chi_0'\left(\frac{y(t,x)}{R}\right)}{R}\left(-\frac{R'}{R}y-\dt g+\frac{C'}{2\pi}\right)\, dx\, dt.
\end{split}
\end{equation*}
Since $y\geq 0$ the term $\chi_0'(\frac{y}{R})$ is non negative and non zero provided  $\frac{1}{2}\leq \frac{y}{R}\leq 1$, so we obtain
\begin{equation*}
\begin{split}
\int_{\RR} \F(T,x) \om^2(T,x)\,dx -\int_{\RR} \F(0,x)
\om^2_0(x)\,dx& \leq \int_0^T \int_{\RR} \om^2
\frac{\chi_0'}{R}\left(-\frac{R'}{2}-\dt g+\frac{C'}{2\pi}\right)\, dx\, dt.
\end{split}
\end{equation*}
Using \eqref{b1} and the explicit from of $C(t)$ leads to
\begin{equation*}
\begin{split}
\int_{\RR} \F(T,x) \om^2(T,x)\,dx -\int_{\RR} \F(0,x)
\om^2_0(x)\,dx& \leq \int_0^T \int_{\RR} \om^2
\frac{\chi_0'}{R}\left(-\frac{R'}{2}+C_2\right)\, dx\, dt
\end{split}
\end{equation*}
for some constant $C_2$.
We now choose
\begin{equation*}
R(t)=C_3+2 C_2 t,
\end{equation*}
with $C_3$ to be determined later on, so that
\begin{equation*}
\int_{\RR} \F(T,x) \om^2(T,x)\,dx \leq \int_{\RR} \F(0,x)
\om^2_0(x)\,dx.
\end{equation*}
Since $|g(0,x)|\leq C_0$ for all $x$ and since $\om_0$ vanishes on
$B(0,R_0)$, we have for all $x\in \supp \om_0$
\begin{equation*}
y(0,x)=-\frac{\ln|x|}{2 \pi}-g(0,x)+\frac{C(0)}{2\pi}\leq -\frac{\ln R_0}{2 \pi} +3C_0:=C_4.
\end{equation*}
 We finally choose $C_3$ so that
 \begin{equation*}
 \frac{C_4}{C_3} \leq \frac{1}{2}.
 \end{equation*}
 For this choice, we have
 \begin{equation*}
 \F(0,x)\om_0^2(x)=\chi_0\left(\frac{y(0,x)}{C_3}\right)\om_0^2(x)\equiv 0.
 \end{equation*}
 We deduce that for all $T$, $\F(T,x)\om^2(T,x)\equiv 0$. We consider $x$ such that
\begin{equation*}
 -\ln|x|\geq 2 \pi R(T)+C_0\left(1+\ln(1+T)\right),
\end{equation*}
then
\begin{equation*}
\frac{y}{R}\geq  \frac{-\ln|x|-g}{2\pi R}\geq 1,
\end{equation*}
therefore $\om(T,x)=0$. So finally
\begin{equation*}
 \om(T)\equiv 0 \qquad \textrm{on \: \:} B(0,e^{-C_0\left(1+\ln(1+T)\right)-2\pi R(T)}),
\end{equation*}
 and the conclusion follows.
\end{proof}

Using Proposition \ref{constant_vorticity_2}, it is then
straightforward to adapt the proof of Theorem \ref{theo} with a
fixed vortex point instead of a moving vortex point and finally
conclude that uniqueness holds for \eqref{s1}.

\begin{theorem}
 Let $\om_0\in L^1\cap L^\infty(\RR)$ be compactly supported such that
\begin{equation*}
\supp(\om_0)\cap \{0\}=\emptyset.
\end{equation*}
Then there exists a unique Eulerian solution of equation \eqref{s1}
with this initial data.
\end{theorem}

\subsection*{Acknowledgements} The authors warmly thank one of the Referees for his very judicious indications concerning the proof of uniqueness via Lagrangian methods. They are also indebted to H.J. Nussenzveig Lopes for having initiated a collaboration between them.

\end{document}